\documentclass[11pt]{amsart}

\usepackage{amsmath,amsthm, amscd, amssymb, amsfonts}
\usepackage{epsfig}

\newcommand{\otb}{{\overline{\otimes}}}
\newcommand{\otk}{{\otimes_{\ku}}}

\newcommand{\Mo}{{\mathcal M}}
\newcommand{\Mor}{{\mathcal M}^{\text{rev}}}
\newcommand{\No}{{\mathcal N}}

\newcommand{\Ss}{{\mathcal S}}
\newcommand{\ot}{{\otimes}}
\newcommand{\Ze}{{\mathcal Z}}

\newcommand{\Ac}{{\mathcal A}}

\newcommand{\ca}{{\mathcal C}}

\newcommand{\Do}{{\mathcal D}}

\newcommand{\Fc}{{\mathcal F}}
\newcommand{\Gc}{{\mathcal G}}

\newcommand{\op}{\rm{op}}
\newcommand{\cop}{\rm{cop}}

\newcommand{\diag}{\,\rm{diag}}

\newcommand{\ku}{{\Bbbk}}

\newcommand{\uno}{{\mathbf 1}}

\newcommand{\id}{\mbox{\rm id\,}}

\newcommand{\invm}{\mbox{\rm InvMod\,}}
\newcommand{\brp}{\mbox{\rm BrPic\,}}

\newcommand{\Mod}{\mbox{\rm Mod\,}}

\newcommand{\Fun}{\operatorname{Fun}}

\newcommand{\bim}{\mbox{\rm Bimod\,}}

\newcommand\Rep{\operatorname{Rep}}

\newcommand\co{\operatorname{co}}

\newcommand\Hom{\operatorname{Hom}}

\newcommand{\End}{\operatorname{End}}

\renewcommand{\_}[1]{\mbox{$_{\left( #1 \right)}$}}

\theoremstyle{plain}
\numberwithin{equation}{section}

\newtheorem{teo}{Theorem}[section]
\newtheorem{lema}[teo]{Lemma}
\newtheorem{cor}[teo]{Corollary}
\newtheorem{prop}[teo]{Proposition}
\newtheorem{claim}{Claim}[section]

\theoremstyle{definition}
\newtheorem{defi}[teo]{Definition}
  
\theoremstyle{remark}

\newtheorem{rmk}[teo]{Remark}

\def\pf{\begin{proof}}
\def\epf{\end{proof}}
\theoremstyle{remark}

\begin{document}

\title[Tensor product of bimodule categories over Hopf algebras ]{On the tensor product of bimodule categories over  Hopf algebras}
\author[ Mombelli]{
Mart\'\i n Mombelli }
\thanks{The work  was  supported by
 CONICET, Secyt (UNC), Mincyt (C\'ordoba) Argentina}
\address{ Facultad de Matem\'atica, Astronom\'\i a y F\'\i sica, \newline
\indent Universidad Nacional de C\'ordoba, \newline \indent Medina Allende s/n,
(5000) Ciudad Universitaria, \newline \indent C\'ordoba, Argentina}
\email{martin10090@gmail.com, mombelli@mate.uncor.edu
\newline \indent\emph{URL:}\/ http://www.mate.uncor.edu/$\sim$mombelli}

\begin{abstract} Let $H$ be a finite-dimensional Hopf algebra. We give a description
of the tensor product of bimodule categories over $\Rep(H)$.  When the bimodule categories
 are invertible this description can be given explicitly. We present some
consequences of this description in the case $H$ is a pointed Hopf algebra.

\bigbreak
\bigbreak
\bigbreak
\bigbreak

\end{abstract}

\date{\today}
\maketitle

\section*{Introduction}
The \emph{Brauer-Picard groupoid} of finite tensor categories, introduced and studied in \cite{ENO}, is the 3-groupoid whose objects are
finite tensor categories, a 1-morphism between two tensor categories $\ca_1, \ca_2$ are
invertible $(\ca_1, \ca_2)$-bimodule categories, 2-morphisms are equivalences of such bimodule
categories and 3-morphisms are isomorphisms of such equivalences. Given a tensor category $\ca$ the \emph{Brauer-Picard group} of $\ca$, denoted by $\text{BrPic}(\ca)$, is the group of equivalence classes of invertible
$\ca$-bimodule categories. 

The Brauer-Picard group of a tensor category  has been used  to classify its extensions
 by a finite group \cite{ENO}. Also it has a close relation to certain  structures
appearing in mathematical physics, see \cite{KK}. In the work  \cite{ENO} the authors compute the Brauer-Picard group of 
categories $Vect_G$ of finite-dimensional $G$-graded vector spaces, where $G$ is an
Abelian group.

\medbreak

It is natural to pursue the  computation of the Brauer-Picard group of
the tensor category of representations of an arbitrary finite-dimensional Hopf algebra $H$.
To compute $\text{BrPic}(\Rep(H))$ one has to be able to give an explicit description of
tensor product of two $\Rep(H)$-bimodule categories. 

It is well-known that any
indecomposable exact $\Rep(H)$-bimodule category is equivalent to ${}_K\Mo$, the category
of finite-dimensional left $K$-modules, where $K$ is a right $H\otk H^{\cop}$-simple left 
$H\otk H^{\cop}$-comodule algebra. If $S$ is another such $H\otk H^{\cop}$-comodule algebra
one could ask about the decomposition of ${}_S\Mo \boxtimes_{\Rep(H)}{}_K\Mo$ in indecomposable
$\Rep(H)$-bimodule categories. For group algebras of finite Abelian groups this
decomposition was explicitly given in \cite{ENO}, but for arbitrary Hopf algebras this problem seems
more complicated. However, if both bimodule categories ${}_K\Mo$, ${}_S\Mo$
are invertible then ${}_S\Mo \boxtimes_{\Rep(H)}{}_K\Mo$ is indecomposable and, under some
additional assumptions, it is equivalent
to ${}_{S\Box_H K}\Mo$. We present some consequences of this result
that will be useful to compute the Brauer-Picard group for pointed Hopf algebras
over an Abelian group.
\medbreak

The contents of the paper are the following. Section \ref{section-rep} is dedicated to  recall necessary definitions
 and facts on representations of tensor categories. 
 In Section \ref{tens-p} we study the tensor product of bimodule categories
over the category $\Rep(H)$, where $H$ is a finite-dimensional Hopf algebra and in Section \ref{tp-q} we restrict to the case when
$H$ is quasi-triangular, allowing us to give another proof of  \cite[Corollary 8.10]{Gr}
concerning about the fusion rules of module categories over a finite group.

\section{Preliminaries and notation}

Throughout the paper $\ku$ will denote an algebraically closed field of characteristic zero. All vector spaces  will be
 considered  over $\ku$. For any Abelian category $\Ac$ we shall denote by $\Ac^{\op}$ the
\emph{opposite Abelian category}, that is objects are the same but arrows are reversed. If $A$ is an algebra
we shall denote by ${}_A\Mo$ the category of finite-dimensional left $A$-modules.

\medbreak

If $H$ is a Hopf algebra we shall denote by $\Ss_H$ its antipode. If  $K, S$ are left $H$-comodule algebras with coaction given by $\lambda_K, \lambda_S $
we shall denote by ${}_K^H\Mo_S$ the category of $(K,S)$-bimodules $V$ equipped with a left
$H$-coaction $\delta:V\to H\ot V$ such that $\delta$ is a morphism of $(K,S)$-bimodules, that is
$$ (k \cdot v\cdot s)\_{-1}\ot (k \cdot v\cdot s)\_0= k\_{-1}v\_{-1} s\_{-1}\ot k\_0
 \cdot v\_0\cdot s\_0,$$
for all $s\in S, k\in K, v\in V$.

\medbreak

If $H$ is a Hopf algebra, $H_0$ is the coradical.  If $(K, \lambda)$ is a left $H$-comodule
algebra and $H_0$ is a Hopf subalgebra, $K_0= \lambda^{-1}(H_0\otk K)$ is  a left
 $H_0$-comodule algebra.

\subsection{Tensor categories} A \emph{ tensor category over} $\ku$ is a $\ku$-linear Abelian rigid monoidal category.
Hereafter all tensor categories will be assumed to be over a field $\ku$. A
\emph{finite tensor category} \cite{eo} is a tensor category such that it has only a finite number of isomorphism
classes of simple objects, Hom spaces are finite-dimensional
$\ku$-vector spaces, all objects have finite lenght, every simple object has a projective cover and
the unit object is simple. All  functors will be assumed to be
$\ku$-linear.

\subsection{Quasi-triangular Hopf algebras}

Let $H$ be a Hopf algebra. A \emph{quasi-triangular structure} on $H$ is an  invertible 
element $R\in H\otk H$ such that 
 
\begin{equation}\label{qt1} (\Delta\ot \id)(R)=R_{13}R_{23}, \quad (\id\ot \Delta)(R)=R_{13}R_{12},
\end{equation}
\begin{equation} \label{qt2} R^1h\_1\ot R^2 h\_2 =h\_2 R^1\ot h\_1 R^2, \text{ for all } h\in H.
\end{equation}

It is well known that $(\Ss_H\ot\id)(R)=R^{-1}=(\id\ot\Ss_H)(R)$. For the inverse of $R$ we shall use the notation
$R^{-1}=R^{-1}\ot R^{-2}$.

If $(H,R)$ is a  quasi-triangular Hopf algebra the category $\Rep(H)$ is braided with braiding
given by $c_{X,Y}:X\otk Y\to Y\otk X$, $c_{X,Y}(x\ot y)= R^2\cdot y \ot R^1\cdot x$ for all $X,Y\in \Rep(H)$,
 $x\in X, y\in Y$. The inverse of $c$ is given by $c^{-1}_{X,Y}(x\ot y)= R^{-1}\cdot y \ot R^{-2}\cdot x$ for all $X,Y\in \Rep(H)$,
 $x\in X, y\in Y$

\section{Representations of tensor categories}\label{section-rep}
Let $\ca$ be a tensor category. A \emph{left representation} of $\ca$, or a \emph{left module category} over $\ca$ is an
 Abelian category $\Mo$ equipped with an exact
bifunctor $\otb: \ca \times \Mo \to \Mo$, that we will sometimes refer as the \emph{action},  natural associativity
and unit isomorphisms $m_{X,Y,M}: (X\otimes Y)\otb M \to X\otimes
(Y\otb M)$, $\ell_M: \uno \otb M\to M$ such that
\begin{equation}\label{left-modulecat1} m_{X, Y, Z\otimes M}\; m_{X\otimes Y, Z,
M}= (\id_{X}\otimes m_{Y,Z, M})\;  m_{X, Y\otimes Z, M}\;  (a_{X,
Y, Z}\otimes \id_{M}),
\end{equation}
\begin{equation}\label{left-modulecat2} (\id_{X}\otimes l_M)m_{X,{\bf
1} ,M}=r_X\otimes \id_{M}.
\end{equation}

A left module category $\Mo$ is \emph{exact} \cite{eo},
  if for any projective object
$P\in \ca$ the object $P\otb M$ is projective in $\Mo$ for all
$M\in\Mo$.  A \emph{right module category }over $\ca$
 is an Abelian category $\Mo$ equipped with an exact
bifunctor $\otb:  \Mo\times  \ca\to \Mo$ equipped with  isomorphisms 
$\widetilde{m}_{M, X,Y}: M\otb (X\ot Y)\to (M\otb X) \otb Y$, $r_M:M\otb \uno\to M$ such that
\begin{equation}\label{right-modulecat1} \widetilde{m}_{M\otb X, Y ,Z }\; \widetilde{m}_{M,X ,Y\ot Z } (\id_M \otb a_{X,Y,Z})=
(\widetilde{m}_{M,X , Y}\ot \id_Z)\, \widetilde{m}_{M,X\ot Y ,Z },
\end{equation}
\begin{equation}\label{right-modulecat2} (r_M\ot \id_X)  \widetilde{m}_{M,\uno, X}= \id_M\ot l_X.
\end{equation}

 A $(\ca, \ca')-$\emph{bimodule category}  is an Abelian category $\Mo$  with left $\ca$-module category
 and right $\ca'$-module category  structure
together with natural
isomorphisms $\{\gamma_{X,M,Y}:(X\otb M) \otb Y:\to X\otb (M\otb Y), X\in\ca, Y\in\ca', M\in \Mo\}$ satisfying 
certain axioms. For details the reader is referred to \cite[Prop. 2.12]{Gr}. A $(\ca, \ca')$-bimodule category
is the same as left $\ca\boxtimes \ca'^{\op}$-module category. Here $\boxtimes$ denotes Deligne's
tensor product of Abelian categories \cite{De}. For a bimodule category $\Mo$
we shall denote by   
$$\{m^l_{X,Y,M}:(X\otimes Y)\otb M \to X\otimes
(Y\otb M: X, Y\in \ca, M\in \Mo\}\, \text{  and }$$
$$ \{m^r_{M, X,Y}: M\otb (X\ot Y)\to (M\otb X) \otb Y:X, Y\in \ca, M\in \Mo\}$$
the left and right associativity isomorphisms respectively.
\medbreak

If $\Mo$ is a   right $\ca$-module category then $\Mo^{\op}$ denotes the
 opposite  Abelian category with left $\ca$ action $\ca\times\Mo^{\op} \to \Mo^{\op}$, $(M, X)\mapsto  M\otb X^*$ and associativity
 isomorphisms $m^{\op}_{X,Y,M}=m^{-1}_{Y^*, X^*, M}$ for all $X, Y\in \ca, M\in \Mo$. Similarly if  $\Mo$ is a
   left $\ca$-module category. If  $\Mo$ is a $(\ca,\Do)$-bimodule category then $\Mo^{\op}$ is a 
$(\Do,\ca)$-bimodule category. See \cite[Prop. 2.15]{Gr}.

\medbreak

A module functor between left $\ca$-module categories $\Mo$ and $\Mo'$ over a
tensor category $\ca$ is a pair $(T,c)$, where $T:\Mo \to
\Mo'$ is a  functor and $c_{X,M}: T(X\otb M)\to
X\otb T(M)$ is a family of  natural isomorphism such that for any $X, Y\in
\ca$, $M\in \Mo$:
\begin{align}\label{modfunctor1}
(\id_X\otimes c_{Y,M})c_{X,Y\otb M}T(m_{X,Y,M}) &=
m_{X,Y,T(M)}\, c_{X\otimes Y,M}
\\\label{modfunctor2}
\ell_{T(M)} \,c_{\uno ,M} &=T(\ell_{M}).
\end{align}
We shall denote this functor by $(T, c): \Mo \to \Mo'$. Sometimes we shall denote the family
of isomorphisms $c^T$ to emphasize the fact that they are related to the functor $T$.

\medbreak  Let $\Mo_1$ and $\Mo_2$ be left $\ca$-module categories. 
The category whose
objects are module functors $(\Fc, c):\Mo_1\to\Mo_2$ will be denoted by $\Fun_{\ca}(\Mo_1, \Mo_2)$. A
morphism between  $(\Fc,c)$ and $(\Gc,d)\in\Fun_{\ca}(\Mo_1,
\Mo_2)$ is a natural transformation $\alpha: \Fc \to \Gc$ such
that for any $X\in \ca$, $M\in \Mo_1$:
\begin{gather}
\label{modfunctor3} d_{X,M}\alpha_{X\otb M} =(\id_{X}\otb \alpha_{M})c_{X,M}.
\end{gather}

  Two module categories $\Mo_1$ and $\Mo_2$ over $\ca$ are {\em equivalent} if there exist module functors $F:\Mo_1\to
\Mo_2$ and $G:\Mo_2\to \Mo_1$ and natural isomorphisms
$\id_{\Mo_1} \to F\circ G$, $\id_{\Mo_2} \to G\circ F$ that
satisfy \eqref{modfunctor3}.

\medbreak

The  direct sum of two module categories $\Mo_1$ and $\Mo_2$ over
a tensor category $\ca$  is the $\ku$-linear category $\Mo_1\times
\Mo_2$ with coordinate-wise module structure. A module category is
{\em indecomposable} if it is not equivalent to a direct sum of
two non trivial module categories. Any exact module category is 
equivalent to a direct sum of indecomposable exact module categories, see \cite{eo}.

If $\Mo,  \Mo'$ are right $\ca$-modules, a module functor from $\Mo$ to $  \Mo'$ is a pair $(T, d)$ where
$T:\Mo \to \Mo'$ is a  functor and $d_{M,X}:T(M\otb X)\to T(M)\otb X$ is a family of isomorphisms
such that for any $X, Y\in
\ca$, $M\in \Mo$:
\begin{align}\label{modfunctor11}
( d_{M,X}\ot \id_Y)d_{M\otb X, Y}T(m_{M, X, Y}) &=
m_{T(M), X,Y}\, d_{M, X\ot Y},
\\\label{modfunctor21}
\ell_{T(M)} \,c_{\uno ,M} &=T(\ell_{M}).
\end{align}
If $\Mo,  \Mo'$ are $(\ca, \Do)$-bimodule categories, a \emph{bimodule functor}
is the same as a module functor of $\ca \boxtimes \Do^{\op}$-module categories,
that is a  functor $F:\Mo\to \Mo'$ such that  $(F,c):\Mo\to \Mo'$ is a functor of left
 $\ca$-module categories,  also $(F,d):\Mo\to \Mo'$ is a functor of right
$\Do$-module categories and
\begin{equation}\label{bimod-funct-def}
(\id_X\ot d_{M,Y}) c_{X,M\otb_r Y} F(\gamma_{X,M,Y})=
\gamma_{X,F(M),Y} (c_{X,M}\ot \id_Y) d_{X\otb_l M, Y},
\end{equation}
for all $M\in  \Mo$, $X\in\ca$, $Y\in \Do$.

\subsection{Tensor product of bimodule categories}\label{t-modcat}

Let  $\ca, \ca', \mathcal{E}, \mathcal{E}'$ be tensor categories. If $\Mo$ is a $(\ca,\mathcal{E})$-bimodule
 category and $\No$ is an $(\mathcal{E},\ca')$-bimodule
category,  the tensor product over $\mathcal{E}$ is denoted by $\Mo\boxtimes_\mathcal{E}\No$. 
This category is a $(\ca,\ca')$-bimodule category. For more details on the tensor product of module categories 
the reader is referred to \cite{ENO}, \cite{Gr}.

\medbreak

If $\Mo$ is a $(\ca, \mathcal{E})$-bimodule category and $\No$ is a $(\ca, \mathcal{E}')$-bimodule category
 then the category $\Fun_\ca(\Mo, \No)$ has a structure of $(\mathcal{E}, \mathcal{E}')$-bimodule category, 
see \cite[Prop. 3.18]{Gr}. Let us briefly describe both
 structures. Let us denote
$$\otb^l: \mathcal{E} \times \Fun_\ca(\Mo, \No)\to \Fun_\ca(\Mo, \No),
\otb^r:\mathcal{E}' \times\Fun_\ca(\Mo, \No)\to \Fun_\ca(\Mo, \No)$$
the left and right actions. If $X\in \mathcal{E}$,  $Y\in  \mathcal{E}'$, $F\in \Fun_\ca(\Mo, \No)$ and $M\in \Mo$, then 
$$(X \otb^l F)(M)=F(M \otb X), \quad (F\otb^r Y)(M)=F(M)\otb Y$$
 The module structures are the following. Let $X, X'\in \mathcal{E}$, $M\in \Mo$ and let  $c^F_{X',M}:F(X'\otb M )\to X'\otb F(M)$
 be the module functor structure of $F$. Then  $c^{X\otb^l F}_{X',M}:(X\otb F)(X'\otb M )\to X'\otb (X\otb F)(M)$
 is defined as the composition
$$F((X'\otb M)\otb X) \xrightarrow{F(\gamma_{X',M,X })} F( X'\otb (M\otb X)) \xrightarrow{c^F_{X',M\otb X}}
 X'\otb F(M\otb X) .  $$
The associativity $m^l_{X,X',F}: (X\ot X')\otb^l F\to X\otb^l (X'\otb^l F)$ is the natural isomorphism
$$ F(M\otb (X\ot X')) \xrightarrow{ F(m^r_{M,X,X'})}  F((M\otb X)\otb X'),$$
for any $X, X'\in \mathcal{E}$, $M\in \Mo$. Also the map
 $$c^{ F\otb^r Y}_{X,M}:( F\otb^r Y)(X\otb M)
\to X\otb ( F\otb^r Y)(M)$$ is defined by the composition
$$F(X\otb  M)\otb Y \xrightarrow{c^F_{X,M}\ot\id_Y} (X\otb F(M))\otb Y \xrightarrow{\gamma_{X,F(M),Y}}
X\otb \big( F(M)\otb Y \big).$$

\begin{prop}\label{funct-mod}\cite[Thm. 3.20]{Gr} If $\Mo$ is a  $(\mathcal{E},\ca)$-bimodule and $\No$ is a  $(\ca, \mathcal{E}')$-bimodule then there is a canonical
 equivalence  of $(\mathcal{E},\mathcal{E}')$-bimodule categories:
\begin{equation}\label{equival-tens-prod-of-modcat}\Mo^{\op}\boxtimes_\ca\No \simeq \Fun_\ca(\Mo, \No).\qed
\end{equation}
\end{prop}

\subsection{The  center of a bimodule category}\label{the-center-bimodcat} The following definition was given in \cite{GNN}. 

\begin{defi}  If $\Mo$ is 
a $\ca$-bimodule category the\emph{ center of }$\Mo$ is the category $\Ze_\ca(\Mo)$
whose objects are pairs $(M, \phi^M)$ where $M\in  \Mo$ and $\{\phi^M_X:X\otb M\to M\otb X: X\in \ca\}$ is a family
of natural isomorphisms such that
\begin{equation}\label{center-modcat} m^r_{M,X,Y} \phi^M_{X\ot Y}= (\phi^M_X\ot\id_Y)\, \gamma^{-1}_{X,M,Y}\,
(\id_X\ot \phi^M_Y) \,m^l_{X,Y,M},
\end{equation}
for all $X, Y\in \ca$, $M\in\Mo$. A morphism between two objects $(M, \phi^M)$, $(N, \phi^N)$ in $\Ze_\ca(\Mo)$ is a morphism
$f:M\to N$ in $\Mo$ such that $(f\ot\id_X) \phi^M_X=\phi^N_X (\id_X\ot f)$ for all $X\in \ca$.
 
\end{defi}

\begin{lema}\cite[Lemma 7.8]{Gr} If $\Mo$ is a $\ca$-bimodule category
 the center $\Ze_\ca(\Mo)$ is a $\Ze(\ca)$-bimodule category .
\end{lema}

Let us briefly explain the left and right actions that we shall denote them by $\otb_l$ and $\otb_r$ respectively.
For any $X\in \ca$, $M\in \Mo$ define $$(X, c_X)\otb_l (M, \phi^M)=(X\otb M,  \phi^{X\otb M})
\text{ and }$$ $$(M, \phi^M)\otb_r  (X, c_X)= (M\otb X, \phi^{M\otb X})$$ where
\begin{equation}\label{lft-bimod-act} \phi^{X\otb M}_Y=\gamma^{-1}_{X,M,Y} (\id_X\ot \phi^M_Y) m^l_{X,Y,M}
(c_{YX}\ot \id_M)  ({m^l_{Y,X,M}})^{-1},
\end{equation}
\begin{equation}\label{rht-bimod-act}  \phi^{M\otb X}_Y={m^r_{M,X,Y}}(\id_M\ot c_{YX})(m^r_{M,Y,X})^{-1}
(\phi^M_Y\ot\id_X) \gamma^{-1}_{Y,M,X},
\end{equation}
for all $Y\in\ca$.

\subsection{Module categories over Hopf algebras}  Assume that $H$ is a finite-dimensional Hopf algebra
 and let $(\Ac,\lambda)$ be a left $H$-comodule algebra.  The
category  ${}_\Ac\Mo$ is a representation
of $\Rep(H)$. The action $\otb:\Rep(H)\times {}_\Ac\Mo\to {}_\Ac\Mo$ is given by
$V\otb M=V\otk M$ for all $V\in \Rep(H)$, $M\in {}_\Ac\Mo$. The left $\Ac$-module
structure on $V\otk M$ is given by the coaction $\lambda$. When $\Ac$ is right $H$-simple, that is,
it has no non-trivial right ideal $H$-costable, then the category ${}_\Ac\Mo$ is exact. 
Reciprocally, if $\Mo$ is an exact indecomposable module category over
 $\Rep(H)$ then
there exists a left $H$-comodule algebra $\Ac$ right $H$-simple with trivial
coinvariants such that $\Mo\simeq {}_\Ac\Mo$ as $\Rep(H)$-modules, see
\cite[Theorem 3.3]{AM}.

\begin{defi}\label{op-comod} If $(\Ac,\rho)$ is a right $H$-comodule algebra then $(\Ac^{\op},\bar{\rho})$
 is a left $H$-comodule algebra,
 where $\Ac^{\op}$ denotes the opposite algebra and $\bar{\rho}:\Ac\to H\ot \Ac$ is defined by
 $\bar{\rho}(a)=\Ss_H(a\_1)\ot a\_0$, where $\rho(a)=a\_0\ot a\_1$ for all $a\in\Ac$. We shall denote this
left $H$-comodule algebra by $\bar{\Ac}$. 
\end{defi}

\begin{lema}\label{comod-opos} There is an equivalence $\big({}_\Ac\Mo\big)^{\op} \simeq {}_{\bar{\Ac}}\Mo$ as left $\Rep(H)$-modules.
\end{lema}
\pf Define $(F,c):\big({}_\Ac\Mo\big)^{\op}\to {}_{\bar{\Ac}}\Mo$ by $F(M)=M^*$ for any $M\in {}_\Ac\Mo$ .
If $f\in M^*, m\in M, a\in  \Ac$ then $(a\cdot f)(m)=f(a\cdot m)$. For any $X\in\Rep(H)$, $M\in \big({}_\Ac\Mo\big)^{\op}$
the maps $c_{X,M}:F(X\otb M)\to X\otb F(M)$ are the identities. One can easily verify that this functor defines an 
equivalence of module categories.
\epf

\begin{prop}\cite[Prop. 1.23]{AM}\label{modfunct-caseHopf}  If $\Ac$ and $\Ac'$ are right $H$-simple
 left $H$-comodule algebras, there is an equivalence  of categories
\begin{equation}\label{modfunct-hopf}
 \Fun_{Rep(H)}({}_\Ac\Mo, {}_{\Ac'}\Mo)\simeq {}_{\Ac'}^H\Mo_\Ac.\qed
\end{equation}
\end{prop} 
We shall explain briefly the proof of this Proposition. 
Any module functor $(F,c^F):{}_\Ac\Mo\rightarrow {}_{\Ac'}\Mo$ is exact \cite{eo}, thus there is exists an object 
$P\in {}_{\Ac'}\Mo_\Ac$ such that $F(M)=P\ot_\Ac M$. The object $P$ has a left $H$-comodule structure given by
$$ \lambda:P\to H\otk P,  \quad \lambda(p)=c^F_{H, \Ac}(p\ot 1\ot 1),$$
for all $p\in P$.

\medbreak

For any finite-dimensional Hopf algebra $H$ we shall denote by $\diag(H)$ the left
$H\otk H^{\cop}$-comodule algebra with $H$ as the underlying algebra structure and comodule
structure:
$$\lambda:\diag(H)\to H\otk H^{\cop}\otk diag(H), \quad \lambda(h)=h\_1\ot h\_3\ot h\_2,$$
for all $h\in H$.
The category ${}_H\Mo$ is a $\Rep(H)$-bimodule category with obvious structure. The proof
of the following result is easy and omitted. 
\begin{lema}\label{diagonal-comod} There is an equivalence of $\Rep(H)$-bimodule categories
$$ {}_H\Mo\simeq {}_{\diag(H)}\Mo.$$\qed
\end{lema}

\section{Tensor product of bimodule categories over Hopf algebras}\label{tens-p}

Let $A, B$ be  finite-dimensional Hopf algebras. A $(\Rep(B), \Rep(A))$-bimo\-dule category is the same as a left
$\Rep(A^{\cop}\ot B)$-module category, see \cite[Prop. 5.5]{De}. Thus any exact $(\Rep(B), \Rep(A))$-bimodule category is equivalent
to the category ${}_S\Mo$ of left $S$-modules, where $S$ is a finite-dimensional  right $A^{\cop}\ot B$-simple left $A^{\cop}\ot B$-comodule
algebra, see \cite[Thm. 3.3]{AM}. The main purpose of this section is to understand the tensor product of $(\Rep(B), \Rep(A))$-bimodule categories.

\medbreak

Set $\pi_A:A\ot B\to A$, $\pi_B:A\ot B\to B$ the algebra maps
$$\pi_A(x\ot y)=\epsilon(y) x, \quad \pi_B(x\ot y)=\epsilon(x) y,$$
for all $x\in A, y\in B.$  If $S$ is a left $A^{\cop}\ot B$-comodule algebra the actions of the tensor categories $\Rep(A)$,
$\Rep(B)$ are as follows. If $M\in {}_S\Mo$, $X\in \Rep(B)$, $Y\in \Rep(A)$ then
$$ X\otb M= X\otk M, \quad M\otb Y=Y\otk M,$$
where the left action of $S$ is:
$$s\cdot (x\ot m)=\pi_B(s\_{-1})\cdot x \ot  s\_0 \cdot m, \quad s\cdot (y\ot m)=\pi_A(s\_{-1})\cdot y \ot  s\_0 \cdot m,$$
for all $s\in S, x\in X, y\in Y, m\in M$. We state the following lemma that will be useful later.
\begin{lema}\label{tech-pi} For any $h\in A\ot B$
 \begin{equation}\label{claim11} \pi_B(h\_{1})\ot  \pi_A(h\_{2})= \pi_B(h\_{2})\ot  \pi_A(h\_{1}).\qed
\end{equation}
\end{lema}

We shall give
to the category ${}_{K}^B\Mo_ {S}$  the following  $\Rep(A)$-bimodule category structure.
If $X, Y\in \Rep(A)$, $P\in {}_{K}^B\Mo_ {S}$ then
$$X\otb^l P= P\ot_S (X\otk S), \quad P\otb^r Y= Y\otk P.$$
Here the left $S$-module structure on $X\otk S$ is given by:
\begin{equation}\label{s-action1} s\cdot (x\ot t)=\pi_A(s\_{-1})\cdot x\ot s\_0t,
\end{equation}

for all $s,t\in S, x\in X$. The object $X\otb^l  P$ belongs to the category ${}_{K}^B\Mo_ {S}$ with the following structure:
$$ r\cdot (p\ot x\ot t)\cdot s= r\cdot p\ot x\ot ts, \quad \delta_1(p\ot x\ot s)= p\_{-1} \pi_B(s\_{-1})\ot p\_0\ot x\ot s\_0,$$
for all $x\in X$, $r\in K, s,t\in S, p\in P$. The object $P\otb^r Y$ belongs to the category ${}_{K}^B\Mo_ {S}$ with the following structure:
$$ r\cdot (y\ot p)\cdot s= \pi_A(r\_{-1})\cdot y\ot r\_0\cdot p\cdot s, \quad \delta_2(y\ot p)=p\_{-1} \ot y\ot p\_0,$$
for all $r\in K, s,t\in S, p\in P$, $y\in Y$. We shall denote the category ${}_{K}^B\Mo_ {S}$ with the above described $\Rep(A)$-bimodule category by
$\Mo(A,B,K,S)$ to emphasize the presence of this extra structure.

\begin{prop} The category $\Mo(A,B,K,S)$ is a  $\Rep(A)$-bimodule category.
\end{prop}
\pf
The map $\delta_1:X\otb^l  P\to B\otk X\otb^l  P$ is well defined. Indeed, if $x\in X$, $ s,t\in S, p\in P$ then
\begin{align*}\delta_1(p\cdot t\ot x\ot s)&=p\_{-1} \pi_B(t\_{-1}s\_{-1})\ot p\_0\cdot t\_0\ot x\ot s\_0\\
&=p\_{-1} \pi_B(t\_{-1}s\_{-1})\ot p\_0\ot \pi_A( t\_0\_{-1})\cdot x\ot t\_0\_0s\_0\\
&=p\_{-1} \pi_B(t\_0\_{-1}s\_{-1})\ot p\_0\ot \pi_A( t\_{-1})\cdot x\ot t\_0\_0s\_0\\
&=\delta_1(p \ot t\cdot(x\ot s)).
\end{align*}
The third equality follows from \eqref{claim11}. It can be proven by a straightforward computation that both objects 
$X\otb^l  P$, $P\otb^r Y$ are in the category ${}_{K}^B\Mo_ {S}$. Let $X, Y\in \Rep(A)$, $P\in {}_{K}^B\Mo_ {S}$, define
$$m^l_{X,Y,P}:(X\otk Y)\otb^l  P\to X\otb^l  (Y\otb^l  P), $$
$$ m^r_{M,X,Y}:P\otb^r (X\otk Y)\to (P \otb^r X)\otb^r Y $$
 by
$$m^l_{X,Y,P}(p\ot x\ot y\ot s)=p\ot y\ot 1\ot x\ot s, \quad m^r_{M,X,Y}(x\ot y\ot p)=y\ot x\ot p,$$
for all $x\in X, y\in Y, p\in P, s\in S$. One can verify easily that both maps belong to the category ${}_{K}^B\Mo_ {S}$
and they satisfy axioms \eqref{left-modulecat1}, \eqref{left-modulecat2} and \eqref{right-modulecat1},
\eqref{right-modulecat2} respectively.  The maps $\gamma_{X,P,Y}:(X\otb^l P)\otb^r Y\to X\otb^l (P \otb^r Y)$, $\gamma(
y\ot p\ot x\ot s)=y\ot p\ot x\ot s$ are morphisms in the category  ${}_{K}^B\Mo_ {S}$ and they satisfy the requirements
of \cite[Prop. 2.12]{Gr}, hence $\Mo(A,B,K,S)$ is a  $\Rep(A)$-bimodule category.\epf

\begin{teo}\label{fun-equiv2} Let $K, S$ be two right $A^{\cop}\otk B$-simple left $A^{\cop}\otk B$-comodule algebras.
The equivalence \eqref{modfunct-hopf} establishes an equivalence
$$\Mo(A,B,K,S)\simeq \Fun_{\Rep(B)}({}_S\Mo, {}_{K}\Mo)$$
 of  $\Rep(A)$-bimodule categories.
\end{teo}

\pf Define $\Phi: \Mo(A,B,K,S) \to \Fun_{\Rep(B)}({}_S\Mo, {}_{K}\Mo)$ by
$$\Phi(P)(N)=P\ot_S N,$$
for all $P\in {}_{R}^B\Mo_S$, $N\in {}_S\Mo$. We shall define on the functor $\Phi$ structures of
left and right $\Rep(A)$-module functor. The natural isomorphisms
$c_{X,P}: \Phi(X\otb^l P)\to X\otb^l \Phi(P)$ are defined by
$$(c_{X,P})_N: \big(P\ot_S (X\otk S) \big)\ot_S N\to P\ot_S (X\otk N ),$$
$$ (c_{X,P})_N(p\ot x\ot s\ot n)=p\ot x\ot s\cdot n,$$
for all $N\in {}_S\Mo$, $p\in P$, $x\in X$, $s\in S$, $n\in N$. The natural isomorphisms 
$d_{P,Y}: \Phi(P\otb Y)\to \Phi(P)\otb Y$ is defined by
$$ (d_{P,Y})_N: (Y\otk P)\ot_S N\to Y\otk (P\ot_S N),  (d_{P,Y})_N(y\ot p\ot n)= y\ot p\ot n,$$
for all  $p\in P$, $y\in Y$, $n\in N$. It is easy to prove that the maps $c_{X,P}, d_{P,Y}$ are well-defined
and make the functor $\Phi$ a left and right $\Rep(A)$-module functor, respectively.
\epf

Using the previous Theorem, equivalence \ref{equival-tens-prod-of-modcat} and Lemma \ref{comod-opos} we obtain:
\begin{cor}\label{tens-bimod}  Let $K$ be a right $A^{\cop}\otk B$-simple left $A^{\cop}\otk B$-comodule algebra and
$L$ a right $B^{\cop}\otk A$-simple left $B^{\cop}\otk A$-comodule algebra. There is
an equivalence of  $\Rep(A)$-bimodule categories:
$$ {}_L\Mo \boxtimes_{\Rep(B)}
 {}_K\Mo \simeq  \Mo(A,B,K,\overline{L}).$$\qed
\end{cor}
Recall that $\overline{L}$ is the opposite algebra of $L$ with left  $A^{\cop}\otk B$-comodule
structure $l\mapsto (\Ss_A^{-1}\ot \Ss_B)(\tau(l\_{-1}))\ot l\_0$ where $l\mapsto l\_{-1}\ot l\_0$
is the left $B^{\cop}\ot A$-comodule structure and
$\tau: B\otk A\to A\otk B$ is the map $\tau(b\ot a)=a\ot b$.

\bigbreak
Keep in mind that $K, S$ are  finite-dimensional left $A^{\cop}\otk B$-comodule algebras. Using the map
$\pi_B:A\otk B\to B$ the algebras $K, S$ are left $B$-comodule algebras, thus $\overline{S}$ is a right
$B$-comodule algebra:
\begin{equation}\label{comodB-inbar} \overline{S}\to \overline{S}\otk B, \quad s\mapsto s\_0\ot \Ss_B^{-1}(\pi_B(s\_{-1})),
\end{equation}

for all $s\in S$. Hence it makes sense to consider the co-tensor product $\overline{S} \Box_B K$. It is clear that $\overline{S} \Box_B K$
is a subalgebra of $\overline{S}\otk K$.
The following result is \cite[Lemma 2.2]{CCMT}. We shall give the proof for the sake of completeness. 
\begin{lema}\label{cotensorp} $S\otk K$ is a left $B$-comodule with coaction $\rho:S\otk K \to B\ot S\otk K$ given by
$$\rho(s\ot k)=\pi_B(s\_{-1})\pi_B(k\_{-1})\ot s\_0\ot k\_0,$$
for all $s\in S$, $k\in K$. Moreover $(S\ot K)^{\co B}=\overline{S} \Box_B K$. 
\end{lema}
\pf Let $ \sum s\ot k\in \overline{S} \Box_B K$. Abusing of the notation, from now on we shall omit the summation symbol. Then
 $$ s\_0\ot \Ss_B^{-1}(\pi_B(s\_{-1}))\ot k
= s\ot \pi_B(k\_{-1})\ot k\_0. $$
Thus we deduce that
$$   \pi_B(s\_0\_{-1}) \ot s\_0\_0\ot \Ss_B^{-1}(\pi_B(s\_{-1}))\ot k =   \pi_B(s_{-1}) \ot s\_0\ot \pi_B(k\_{-1})\ot k\_0.$$
Then
\begin{align*} \rho(s\ot k)=\pi_B(s\_0\_{-1})\Ss_B^{-1}(\pi_B(s\_{-1}))\ot  s\_0\_0\ot k=1\ot s\ot k.
\end{align*}
Now, let $s\ot k\in (S\ot K)^{\co B}$ then
$$ 1\ot s\ot k = \pi_B(s\_{-1})\pi_B(k\_{-1})\ot s\_0\ot k\_0.$$
From this equality we deduce that
$  1\ot s\_0\ot\Ss_B^{-1}( \pi_B(s\_{-1})) \ot k$ is equal to $$ \pi_B(s\_{-1})\pi_B(k\_{-1})\ot s\_0\_0\ot \Ss_B^{-1}( \pi_B(s\_0\_{-1}))\ot k\_0.$$
Then $ s\_0\ot \Ss_B^{-1}(\pi_B(s\_{-1}))\ot k
= s\ot \pi_B(k\_{-1})\ot k\_0, $ so $s\ot k\in  \overline{S} \Box_B K$.

\epf

Define the map $\lambda: \overline{S} \Box_B K\to A\otk A^{\cop}\ot \overline{S} \Box_B K$ by
$$\lambda(s\ot k)= \Ss_A(\pi_A(s\_{-1}))\ot \pi_A(k\_{-1}) \ot s\_0\ot k\_0,$$
for all $s\ot k\in \overline{S} \Box_B K$.
\begin{lema}\label{cotensor-comodalg} $(\overline{S} \Box_B K, \lambda)$ is a left $A\otk A^{\cop}$-comodule algebra.
\end{lema}
\pf Let us prove first that $\lambda$ is well-defined. Let $s\ot k\in \overline{S} \Box_B K$, then
$ s\_0\ot \Ss_B^{-1}(\pi_B(s\_{-1}))\ot k
= s\ot \pi_B(k\_{-1})\ot k\_0, $ hence
\begin{align*} s\_0\_0\ot s\_0\_{-1}\ot \Ss_B^{-1}(\pi_B(s\_{-1}))&\ot k\_{-1}\ot k\_0=\\
&=s\_0\ot s\_{-1}\ot \pi_B(k\_{-1})\ot k\_0\_{-1}\ot k\_0\_0.
\end{align*}
Therefore 
\begin{align*} s\_0\_0\ot& \Ss_A(\pi_A(s\_0\_{-1}))\ot \Ss_B^{-1}(\pi_B(s\_{-1}))\ot \pi_A(k\_{-1})\ot k\_0=\\
&=s\_0\ot \Ss_A(\pi_A(s\_{-1}))\ot \pi_B(k\_{-1})\ot \pi_A(k\_0\_{-1})\ot k\_0\_0.
\end{align*}
Thus $\Ss_A(\pi_A(s\_{-1}))\ot  \pi_A(k\_0\_{-1})\ot s\_0\ot \pi_B(k\_{-1})\ot k\_0\_0$ is equal to
\begin{align*} \Ss_A(\pi_A(s\_0\_{-1}))\ot  \pi_A(k\_{-1})\ot s\_0\_0\ot\Ss_B^{-1}(\pi_B(s\_{-1}))\ot  k\_0.
\end{align*}
Using \eqref{claim11} we get that $\Ss_A(\pi_A(s\_{-1}))\ot  \pi_A(k\_0\_{-1})\ot s\_0\ot \pi_B(k\_{-1})\ot k\_0\_0$
 is equal to $\Ss_A(\pi_A(s\_{-1}))\ot  \pi_A(k\_{-1})\ot s\_0\_0\ot\Ss_B^{-1}(\pi_B(s\_0\_{-1}))\ot  k\_0.$ Hence
$\lambda( \overline{S} \Box_B K)\subseteq  A^{\cop}\ot A\ot\overline{S} \Box_B K$. It follows straightforward 
 that $\lambda$ is an algebra map. \epf

Lemma \ref{cotensorp} implies that the category ${}_{\overline{S} \Box_B K}\Mo$ is a $\Rep(A)$-bimodule category. 
In what follows we shall  study the relation between this  $\Rep(A)$-bimodule category and $\Mo(A,B,K,S)$.

\medbreak

Let $\Fc: {}_{\overline{S} \Box_B K}\Mo\to \Mo(A,B,K,S)$, $\Gc:\Mo(A,B,K,S)\to
{}_{\overline{S} \Box_B K}\Mo $ be the functors defined by
\begin{equation}\label{functor-f}
  \Fc(N)=(\overline{S}\otk K)\ot_{\overline{S} \Box_B K}  N, \quad \Gc(P)=P^{\co B},
\end{equation}
for all $N\in {}_{\overline{S} \Box_B K}\Mo$, $P\in \Mo(A,B,K,S)$. This pair of functors were
 considered first in \cite{D}, see also
\cite{CCMT}. The $(K,S)$-bimodule structure on $ \Fc(N)$ is given as follows:
$$ k' \cdot (s\ot k\ot n)\cdot s' = ss'\ot k' k\ot n, \quad \text{for all } s,s'\in S, k,k'\in K, n\in N.$$
Define the map $\delta: \Fc(N)\to B\otk \Fc(N)$ by
$$ \delta(s\ot k\ot n)=\pi_B(k\_{-1})\pi_B(s\_{-1})\ot s\_0\ot k\_0\ot n,$$
for all $s\in S, k\in K, n\in N.$ It follows from $(S\ot K)^{\co B}=\overline{S} \Box_B K$ that  $\delta$ is well-defined. 
Also $ \Fc(N)\in \Mo(A,B,K,S)$, details are left to the reader. The action of $\overline{S} \Box_B K$
on $\Gc(P)$ is given by
$$ (s\ot k)\cdot p= k \cdot p\cdot s,\quad \text{ for all }s\ot k\in \overline{S} \Box_B K.$$

\begin{prop}\label{mod-funct-coprod} The functors $\Fc, \Gc $
are left and right $\Rep(A)$-module
functors.
\end{prop}
\pf First we shall prove that $\Fc$ is a module functor. Let $N\in {}_{\overline{S} \Box_B K}\Mo$ and $X\in \Rep(A)$. Define
$$c_{X,N}: (\overline{S}\otk K)\ot_{\overline{S} \Box_B K} (X\otk N )\to
 \big((\overline{S}\otk K)\ot_{\overline{S} \Box_B K}  N\big) \ot_S (X\otk S)$$
by $c_{X,N}(s\ot k\ot x\ot n)= 1\ot k\ot n\ot x\ot  s$, for all $x\in X$, $s\in S, k\in K, n\in N.$
\begin{claim} The map $c_{X,N}$ is well-defined.
\end{claim}
\begin{proof}[Proof of claim] First observe that for any $x\in X, s,t\in S$ we have that
\begin{align}\label{s-action2} x\ot st= s\_0\cdot (\Ss_A(\pi_A(s\_{-1}))\cdot x\ot t).
\end{align}
Recall that the action of $S$ on $X\otk S$ is given in \eqref{s-action1}.
Let $s'\ot k'\in \overline{S} \Box_B K$,  $x\in X$, $s\in S, k\in K, n\in N.$ Then
\begin{align*}c_{X,N}((s\ot k)\cdot (s'\ot k')\ot x\ot n)&= c_{X,N}(s's\ot kk'\ot x\ot n)=1\ot kk'\ot n\ot x\ot s's\\
&=1\ot kk'\ot n\ot {s'}\_0 \cdot( \Ss_A(\pi_A({s'}\_{-1}))\cdot x\ot s)\\
&={s'}\_0 \ot kk'\ot n\ot  \Ss_A(\pi_A({s'}\_{-1}))\cdot x\ot s
\end{align*}
The second equality follows from \eqref{s-action2}.  On the other hand the element
 $c_{X,N}(s\ot k\ot  \Ss_A(\pi_A({s'}\_{-1}))\cdot x\ot ({s'}\_0\ot k')\cdot n)$ is equal to
\begin{align*}&=1\ot k\ot ({s'}\_0\ot k')\cdot n\ot \Ss_A(\pi_A({s'}\_{-1}))\cdot x\ot s\\
&={s'}\_0 \ot kk'\ot n\ot  \Ss_A(\pi_A({s'}\_{-1}))\cdot x\ot s.
\end{align*}
This finishes the proof of the claim.
\end{proof}
Clearly $c_{X,N}$ is a $(K,S)$-bimodule homomorphism and also a $B$-comodule homomorphism. 
Equations \eqref{modfunctor1} and \eqref{modfunctor2} are  satisfied. 
Thus $(\Fc, c)$ is a module functor.

If $N\in {}_{\overline{S} \Box_B K}\Mo$ and $Y\in \Rep(A)$ define 
$$d_{N,Y}:  (\overline{S}\otk K)\ot_{\overline{S} \Box_B K} (Y\otk N )\to
 Y\otk\big((\overline{S}\otk K)\ot_{\overline{S} \Box_B K}  N\big) $$
by $d_{N,Y}(s\ot k\ot y\ot n)=\pi_A(k\_{-1})\cdot y\ot s\ot k\_0\ot n$ for all
$s\in S, k\in K, n\in N, y\in Y$. It follows from a straightforward computation that the maps
 $d_{N,Y}$ are well-defined isomorphisms in the category $ \Mo(A,B,K,S)$ and they
satisfy equations \eqref{modfunctor11}, \eqref{modfunctor21}. Hence $(\Fc, d)$ is a
module functor.

\medbreak

Now, let us prove that $\Gc$ is a module functor.  Let $P\in  \Mo(A,B,K,S)$, $X, Y\in \Rep(A)$.
Set $c'_{X,P}:\big( P\ot_S X\otk S\big)^{\co B}\to X\otk P^{\co B}$ the map defined by
$$c'_{X,P}(p\ot x\ot s)= \Ss_A(\pi_A(s\_{-1}))\cdot x\ot p\cdot s\_0,$$
for all $p\ot x\ot s\in \big( P\ot_S X\otk S\big)^{\co B}$. Define also $d'_{P,Y}:
:\big( Y\otk P\big)^{\co B}\to Y\otk P^{\co B}$ by
$$ d'_{P,Y}(y\ot p)=y\ot p, \; \text{ for all } y\ot p\in \big( Y\otk P\big)^{\co B}.$$
One can easily prove that $(\Gc, c')$ is a module functor of left $\Rep(A)$-module categories and
 $(\Gc, d')$ is a module functor of right $\Rep(A)$-module categories.
\epf

\section{Tensor product of module categories over a quasi-triangular Hopf algebra}\label{tp-q}

 In this section
$H$ will denote  a finite-dimensional quasi-triangular Hopf algebra. We shall describe
 the tensor product of module categories over $\Rep(H)$.

\subsection{Module categories over a braided tensor category}\label{modcat-over-braided}

First, let us recall some general considerations about the tensor product of module categories
over a braided tensor category. Let $\ca$ be a braided tensor category with braiding $c_{X,Y}:X\ot Y\to Y\ot X$
 for all $X, Y\in\ca$. Let $\Mo$ be a 
right $\ca$-module. Then $\Mo$ has a left $\ca$-module structure $ \ca\times\Mo\to \Mo$ given by $Y\otb^{rev}M :=M\otb Y$
for all $Y\in \ca$, $M\in \Mo$ and the associativity constraints $ m^{rev}_{X,Y,M}:(X\ot Y)\otb^{rev} M\to
   X\otb^{rev} ( Y \otb^{rev} M)$ are given by
$$m^{rev}_{X,Y,M}
=m_{M,Y,X} (\id_M\ot c_{X,Y}),$$
for all $X, Y\in \ca$, $M\in \Mo$. This category is indeed a left $\ca$-module category, see \cite[Lemma 7.2]{Gr},
that we shall denote by $\Mor$.
Equipped with these two structures $\Mo$ is a $\ca$-bimodule category. For details see \cite[Prop. 7.1]{Gr}.

\begin{rmk}\label{mod-o-fun-braided} In particular if $\Mo$ is a right $\ca$-module and $\No$ are left 
$\ca$-module then
$\Mo$ is a bimodule category using the reverse right action, and the tensor product $\Mo \boxtimes_\ca \No$
is a left $\ca$-module category. 
\end{rmk}

If $\Mo$ is a $\ca$-bimodule category then the center $\Ze_\ca(\Mo)$ has two left $\Ze(\ca)$-module
 structures: the 
one denoted by $\otb_l$ explained in section \ref{the-center-bimodcat} given by equation \eqref{rht-bimod-act} and the reverse action
of the right action $\otb_r$  presented in \eqref{lft-bimod-act}. Both right actions give the same module category. 
This result will be useful later.

\begin{prop}\label{proposition-rev} There is an equivalence of left $\Ze(\ca)$-module categories
 between $(\Ze_\ca(\Mo), \otb_{r}^{rev})$ and
 $(\Ze_\ca(\Mo), \otb_l)$.
\end{prop}
\pf Define $(\Fc, d):(\Ze_\ca(\Mo), \otb_{1})\to (\Ze_\ca(\Mo), \otb_r^{rev})$ the module functor as follows. The functor
$\Fc$ on objects is the identity, that  is $\Fc(M,\phi^M)=(M,\phi^M)$ for any 
$(M,\phi^M)\in \Ze_\ca(\Mo)$. If $(X, c_X)\in \Ze(\ca)$ define
$$ d_{X,M}: (X, c_X)\otb_l  (M,\phi^M) \to (M,\phi^M)\otb_r(X, c_X), \;\; d_{M,X}=\phi^M_X.$$
The maps $\phi^M_X$ are morphisms in the category $\Ze_\ca(\Mo)$. Indeed, we must prove that for all
$X, Y\in \ca$, $M\in \Mo$
\begin{equation}\label{phi-are-morphisms} (\phi^M_X\ot \id_Y) \phi^{X\otb M}_Y =  \phi^{ M\otb X}_Y (\id_Y
\ot \phi^M_X).
\end{equation}
Using \eqref{center-modcat} one can see that the right hand side of equation \eqref{phi-are-morphisms} equals to
\begin{equation}\label{techn-01} m^r_{M,X,Y}(\id_M\ot c_{YX}) \phi^M_{Y\ot X}(m^l_{Y,X,M})^{-1}, \end{equation}
and the left hand side of equation \eqref{phi-are-morphisms} equals to
\begin{equation}\label{techn-02} m^r_{M,X,Y}  \phi^M_{X\ot Y} ( c_{YX}\ot \id_M) (m^l_{Y,X,M})^{-1}.
\end{equation}
Follows from  the naturality of $\phi$ that the expressions \eqref{techn-01}, \eqref{techn-02} are equal.
 Let us prove now
that the functor $(\Fc, d)$ is a module functor. Equation  \eqref{modfunctor1} amounts to 
\begin{equation}\label{mod-f-c} (\phi^M_Y\ot\id_X)  \phi^{Y\otb M}_X m^l_{X,Y,M} 
= m^r_{M,Y,X}(c_{X,Y}\ot\id_M) \phi^M_{X\ot Y},
\end{equation}
for all $X,Y\in\ca$, $M\in\Mo$. Equation \eqref{mod-f-c} can be checked by a direct computation.
\epf

\subsection{Tensor product of module categories over a quasi-triangular Hopf algebra}

Let $H$ be a finite-dimensional quasi-triangular Hopf algebra with R-matrix $R$. Any left $\Rep(H)$-module category
is a $\Rep(H)$-bimodule category as explained in the beginning of Section  \ref{modcat-over-braided}.
Given two left $H$-comodule algebras $K, S$
 our aim now is to describe the left $\Rep(H)$-module category over the tensor product 
$_K\Mo\boxtimes_{\Rep(H)} {}_S\Mo$ using the
 left module category $\Fun_{Rep(H)}({}_S\Mo, {}_{K}\Mo)$
 and  Proposition \ref{modfunct-caseHopf}.

\begin{prop}\label{mod over bimod} Let $K, S$ be two left $H$-comodule algebras. 
The category ${}_{K}^H\Mo_S$ is a left $\Rep(H)$-module.
\end{prop}
\pf Define $\otb:\Rep(H)\times {}_{K}^H\Mo_S\to {}_{K}^H\Mo_S$ by
$$ X\otb P:= P\ot_S( X\otk S),$$
for all $X\in \Rep(H)$, $P\in {}_{K}^H\Mo_S$. Here the left action of $S$ on $X\otk S$ is
given by the coaction of $S$. The object $P\ot_S( X\otk S)\in {}_{R}^H\Mo_S$ 
with structure given by
$$\delta_P(p\ot x\ot s)= p\_{-1} R^{2}s\_{-1}\ot p\_0 \ot R^{1} \cdot x\ot s\_{0}, $$
$$ r\cdot (p\ot x\ot s)=  r \cdot p\ot x\ot s, \quad (p\ot x\ot s)\cdot l=p\ot x\ot sl , $$
for all $p\in P$, $r\in K$, $s, l\in S$. Follows straightforward that these maps are well defined. 
Clearly $P\ot_S( X\otk S)$ is a $(K,S)$-bimodule and $\delta_P$ is a $K$-module
morphism. The associativity isomorphisms 
$$m_{X,Y,P}:P\ot_S(X\otk Y)\otk S\to \big(P\ot_S (Y\otk S) \big)\ot_S X\otk S $$ are given by
$$m_{X,Y,P}(p\ot (x\ot y)\ot s)=(p\ot R^{-1}\cdot y\ot 1)\ot (R^{-2}\cdot x\ot s),$$
for all $p\in P, x\in X, y\in Y, s\in S$.
The maps $m_{X,Y,P}$ are well defined morphisms in the category ${}_{K}^H\Mo_S$. Indeed, let $l\in S$ then 
\begin{align*} m_{X,Y,P}(p\ot l\_{-1} \cdot (x\ot y)\ot l\_0s)&=p\ot R^{-1}l\_{-1}\cdot y\ot 1\ot R^{-2}l\_{-2}\cdot x\ot l\_0 s\\
&=p\ot l\_{-2}R^{-1}\cdot y\ot 1\ot l\_{-1}R^{-2}\cdot x\ot l\_0 s\\
&=p\ot l\_{-1}R^{-1}\cdot y\ot l\_0\ot R^{-2}\cdot x\ot  s\\
&=p\cdot l\ot R^{-1}\cdot y\ot 1\ot R^{-2}\cdot x\ot  s\\
&=m_{X,Y,P}(p\cdot l\ot x\ot y\ot s).
\end{align*}
This proofs that $m_{X,Y,P}$ is well-defined. The proof that $m_{X,Y,P}$ is a $(K,S)$-bimodule
morphism is straightforward. 
Let us prove that $m_{X,Y,P}$ is a comodule map. If $\widetilde{P}=P\ot_S (Y\otk S) $ then
$\delta_{\widetilde{P}}(m_{X,Y,P}(p\ot (x\ot y)\ot s))$ equals to
\begin{align*} p\_{-1}J^2r^2 s\_{-1}\ot p\_0\ot J^1R^{-1}\cdot 
y\ot 1
\ot r^1R^{-2}\cdot x\ot s\_0,
\end{align*}
for any $p\in P$, $x\in X$, $y\in Y$, $s\in S$. Here   $R=R^1\ot R^2=J^1\ot J^2=r^1\ot r^2$. On the other
 hand $  (\id_H\ot m_{X,Y,P}) \delta_P(p \ot (x\ot y)\ot s)$ is 
equal to 
\begin{align*}&= p\_{-1}R^2s\_{-1}\ot m_{X,Y,P}(
p\_0\ot R^1\_1\cdot x\ot R^1\_2\cdot y\ot s\_0)\\&= p\_{-1}R^2s\_{-1}\ot p\_0\ot r^{-1} R^1\_2\cdot y\ot 1\ot r^{-2} R^1\_1\cdot x
\ot s\_0\\
&=  p\_{-1}R^2s\_{-1}\ot p\_0\ot  R^1\_1r^{-1}\cdot y\ot 1\ot R^1\_2r^{-2} \cdot x
\ot s\_0.
\end{align*}
The third equality follows from \eqref{qt2}. Both terms are equal if and only if
$$  J^1R^{-1}\ot  r^1R^{-2}  \ot J^2r^2= R^1\_1r^{-1}\ot     R^1\_2r^{-2}    \ot R^2,$$
and this follows by \eqref{qt1}. The associativity of $m$ follows from the Yang-Baxter equation:
$R_{12}R_{13}R_{23}=R_{23}R_{13}R_{12}$ .
\epf

We shall denote the category ${}_{K}^H\Mo_S$ with the structure of left $\Rep(H)$-module
category explained in Proposition \ref{proposition-rev} by $\Mo(R,K,S)$ to emphasize the
fact that the R-matrix in involved in the module category structure.
\begin{teo}\label{fun-equiv} Let $K, S$ be two right $H$-simple left $H$-comodule algebras.
The equivalence \eqref{modfunct-hopf} establishes an equivalence
$$\Mo(R,K,S)\simeq \Fun_{Rep(H)}({}_S\Mo, {}_{K}\Mo)$$
 of  $\Rep(H)$-modules.
\end{teo}
\pf Define $(\Phi, c): \Mo(R,K,S)\to \Fun_{Rep(H)}({}_S\Mo, {}_{K}\Mo)$ by
$$\Phi(P)(N)=P\ot_S N$$
for all $P\in {}_{K}^H\Mo_S$, $N\in {}_S\Mo$. The natural transformations
$c_{X,P}: \Phi(X\otb P)\to X\otb \Phi(P)$ are defined by
$$\big(c_{X,P}\big)_N: \big( P\ot_S (X\otk S)\big)\ot_S N\to P\ot_S (X\otk N),$$
$$\big(c_{X,P}\big)_N(p\ot x\ot s\ot n)=p\ot x \ot s\cdot n,$$
for all $X\in\ca$, $P\in \Mo(R,K,S)$, $N\in {}_S\Mo$, $x\in X, p\in P, n\in N, s\in S$. The functor
$(\Phi, c)$ is a module functor and is an equivalence of module categories.
\epf

\begin{cor} There is an equivalence of left $\Rep(H)$-modules:
 \begin{equation}\label{tensor-prod-modcat}  \big({}_S\Mo \big)^{\op}\boxtimes_{\Rep(H)}
 {}_K\Mo \simeq  \Mo(R,K,S).
 \end{equation}\qed
\end{cor}

\subsection{Fusion rules for $\Rep(\ku G)$-modules}

Let $G$ be a finite group. Using the equivalence  \eqref{tensor-prod-modcat}
 we can give another proof of \cite[Corollary 8.10]{Gr} concerning about the tensor product of indecomposable
exact module categories over $\Rep(\ku G)$. The Hopf algebra $\ku G$ is quasi-triangular with trivial R-matrix $1\ot 1$.

\medbreak

For any subgroup $F\subseteq G$ and $\psi\in Z^2(F, \ku^\times)$ the twisted group algebra $\ku_\psi F$
is a right $\ku G$-simple left $\ku G$-comodule algebra. Let $F_i\subseteq G$ be subgroups and $\psi_i\in Z^2(F_i, \ku^\times)$
for $i=1,2$. Let $S\subseteq G$ be a set of representative classes of the double cosets $F_2\diagdown G\diagup F_1$.
For any $s\in S$ define $F_s= s^{-1}F_1s\cap  F_2 $ and $\psi_s\in  Z^2(F_s, \ku^\times)$ the 2-cocycle
defined by
$$\psi_s(x,y)=\psi_1(sxs^{-1},sys^{-1})\, \psi_2(x , y ),$$
for any $x,y\in F_s$. 
\begin{prop}\label{tensor-prod-modcat-group} \cite[Corollary 8.10]{Gr}  There is an equivalence 
 \begin{equation}\label{tens-prod-g} {}_{\ku_{\psi_1}F_1}\Mo \boxtimes_{\Rep(\ku G)}
{}_{\ku_{\psi_2}F_2}\Mo \simeq \bigoplus_{s\in S} \, {}_{\ku_{\psi_s}F_s}\Mo
  \end{equation}
\end{prop}

\begin{proof} Let $V\in  {}_{\ku_{\psi_2}F_2}^{\ku G}\!\Mo_{\overline{\ku_{\psi_1}F_1}}$ with
coaction given by $\delta:V\to \ku G\ot V$. Then
 $V=\oplus_{g\in G} V_g$ where $V_g=\{v\in V: \delta(v)=g\ot v\}$. For any $s\in S$ define
$$ V\_s=\oplus_{g\in F_1 s F_2} \, V_g,$$
thus $V=\oplus_{s\in S}\,  V\_s$ and each vector space $V\_s$ is a subobject of $V$ in 
the category ${}_{\ku_{\psi_2}F_2}^{\ku G}\!\Mo_{\overline{\ku_{\psi_1}F_1}}$.

The subspace $ V_s$ carries a structure of $\ku_{\psi_s}F_s$ as follows. For any $h\in F_s$, $v\in V_s$ 
define
$$ h \triangleright v= h\cdot ( v\cdot sh s^{-1}).$$
Define the functor  $\Fc:{}_{\ku_{\psi_2}F_2}^{\ku G}\!\Mo_{\overline{\ku_{\psi_1}F_1}}\longrightarrow
\bigoplus_{s\in S} \, {}_{\ku_{\psi_s}F_s}\Mo$, $\Fc(V)=\oplus_{s\in S}\,  V_s$ and for any $s\in S$ the vector space 
$V_s$  has the action
of $\ku_{\psi_s}F_s$ as explained before. The functor $\Fc$ is indeed a module functor. \smallbreak

Let $V\in  {}_{\ku_{\psi_2}F_2}^{\ku G}\!\Mo_{\overline{\ku_{\psi_1}F_1}}$ and assume that
 $V=V\_s$ for some $s\in S$. It is not difficult to see that  
$$(X\otb V)\_s=X\otb V=V\ot_{\overline{\ku_{\psi_1}F_1}} (X\otk\, \overline{\ku_{\psi_1}F_1} )$$ for any
 $X\in \Rep(\ku G)$,
hence 
$$\Fc(X\otb V)=\oplus_{f\in F_1} V_{sf}\ot_{\ku_{\psi_1}F_1} \big(X\otk \,\ku_{\psi_1}F_1\big)$$ as vector spaces. Define
$c_{X, V}:\Fc(X\otb V)\to X\otk\Fc(V)$ by
$$ c_{X, V}(v\ot x\ot f)=f \cdot x\ot v\cdot f,$$
for any $x\in X, v\in V, f\in F_1$. It follows from a straightforward computation that the map $c_{X, V}$ is
well-defined and  equations \eqref{modfunctor1}, 
\eqref{modfunctor2}
are satisfied. Now, define $\Gc: \bigoplus_{s\in S} \, {}_{\ku_{\psi_s}F_s}\Mo \to{}_{\ku_{\psi_2}F_2}^{\ku G}\!\Mo_{\overline{\ku_{\psi_1}F_1}} $
as follows. If $W\in {}_{\ku_{\psi_s}F_s}\Mo$ for some $s\in S$ then
$$\Gc(W)= (\ku F_1\otk \ku F_2)\ot_{\ku_{\psi_s}F_s} W.$$
The right action of $\ku_{\psi_s}F_s$ on the tensor product $\ku F_1\otk \ku F_2$ is 
$$ (x\ot y )\cdot f = \psi_1(x^{-1},sfs^{-1})\psi_2(y,  f )\; s^{-1}f^{-1} sx \ot   yf,$$
for all $x\in F_1, y\in F_2, f\in F_s$. 

For any $x,f\in F_1,$ $ y,g\in F_2$, $w\in W$ define
$$g\cdot (x\ot y\ot w)=\psi_2(g,y)\,(x\ot gy\ot w), $$ $$  (x\ot y\ot w) \cdot f= \psi_1(f, x^{-1})\,   (xf^{-1}\ot y\ot w),$$
$$ \delta(x\ot y\ot w)= ysx\, \ot (x\ot y\ot w).$$
Equipped with these maps the object $\Gc(W)$ is an object in the category 
${}_{\ku_{\psi_1}F_2}^{\ku G}\!\Mo_{\overline{\ku_{\psi_2}F_1}}$.
\end{proof}

\section{Applications for computing the Brauer-Picard group}\label{bpg}

\subsection{The Brauer-Picard group of a tensor category}

Let $\ca_1, \ca_2$ be finite tensor categories. The following definitions were given 
 in \cite{ENO}.

\begin{defi}\begin{itemize}
             \item[(a)] An exact $(\ca_1,\ca_2)$-bimodule category $\Mo$ is \emph{invertible} if there are bimodule equivalences
$$ \Mo^{\op} \boxtimes_{\ca_1} \Mo \simeq \ca_2, \quad  \Mo \boxtimes_{\ca_2} \Mo^{\op} \simeq \ca_1.$$
\item[(b)] The Brauer-Picard groupoid $\underline{\underline{\text{BrPic} }}$ 
is the 3-groupoid whose objects are finite tensor categories, 1-morphisms from
$\ca_1$ to $\ca_2$ are invertible $(\ca_1,\ca_2)$-bimodule categories, 2-morphisms are equivalences
 of such bimodule categories, and 3-morphisms are isomorphisms
of such equivalences. Forgetting the 3-morphisms and the 2-morphisms and identifying 1-morphisms one obtains
the groupoid BrPic. The group $\text{BrPic}(\ca)$ of automorphisms of $\ca$ in BrPic is called the\emph{ Brauer-Picard group
of }$\ca$.

            \end{itemize}

\end{defi}
\subsection{Invertible module categories over a braided tensor category}
Let $\ca$ be a braided tensor category. Any left $\ca$-module category is a $\ca$-bimodule category using the reverse action as
explained in section \ref{modcat-over-braided}. 
\begin{defi} We shall say that an exact $\ca$-module category $\Mo$ is \emph{invertible} if there 
is a bimodule equivalence
$$ \Mo^{\op}\boxtimes_\ca \Mo\simeq \ca.$$
The group of invertible  $\ca$-module categories will be denoted by $\invm(\ca)$
\end{defi}

\begin{prop} Let $\ca$ be a  tensor category. There is an isomorphism of groups $\brp(\ca)\simeq \invm(\Ze(\ca)).$
\end{prop}
\pf Denote by $\Ze: \bim(\ca)\to \Mod(\ca)$ the center functor. As a consequence of \cite[Thm. 7.13, Lemma 7.14]{Gr}
 and Proposition \ref{proposition-rev} this functor restricts to an isomorphism of groups. \epf

\subsection{Invertible $\Rep(H)$-bimodule categories}

In this section we study the tensor product of invertible module categories over the representation
categories of Hopf algebras using the tools developed in the previous sections.
\medbreak

Let $H$ be a finite-dimensional Hopf algebra. Recall that if  $\Mo$ is a
$\Rep(H)$-bimodule category, then there exists a left $H\otk H^{\cop}$-comodule algebra $K$, right
$H\otk H^{\cop}$-simple with trivial coinvariants such that $\Mo\simeq {}_K\Mo$ as  $\Rep(H)$-bimodule categories.

\begin{teo}\label{invert-comod} Let $K, S$ be left $H\otk H^{\cop}$-comodule algebras 
right $H\otk H^{\cop}$-simple with trivial coinvariants. Assume also that
\begin{itemize}
 \item[(i)] $S\otk K$ is free as a left $S \Box_H K$-module,
 \item[(ii)] the module category ${}_{S\Box_H K}\Mo$ is exact, and
 \item[(iii)] ${}_S\Mo, {}_K\Mo$ are invertible $\Rep(H)$-bimodule categories.
\end{itemize}
Then, there is an equivalence of 
$\Rep(H)$-bimodule categories
\begin{equation}\label{pro-cot-a}{}_S\Mo \boxtimes_{\Rep(H)}
 {}_K\Mo
 \simeq {}_{S\Box_H K}\Mo.
\end{equation}
\end{teo}
\pf 
By Corollary \ref{tens-bimod} there exists an equivalence of $\Rep(H)$-bimodule
categories 
$$ {}_S\Mo \boxtimes_{\Rep(H)}
 {}_K\Mo
 \simeq  \Mo(H,H,K,\overline{S}).$$
Since invertible bimodule categories are indecomposable, then the category $\Mo(H,H,K,\overline{S})$ is an indecomposable bimodule category. 
Consider the functor $\Fc:{}_{S\Box_H K}\Mo \to  \Mo(H,H,K,\overline{S})$ explained in
\eqref{functor-f}. Since $S\otk K$ is free as a left $S\Box_H K$-module then
$\Fc$ is full and faithful. The full subcategory of $\Mo(H,H,K,\overline{S})$ consisting of objects $\Fc(N)$
where $N \in {}_{S\Box_H K}\Mo$ is an exact submodule category and since
$\Mo(H,H,K,\overline{S})$ is indecomposable, $\Fc$ must be an equivalence, see
\cite[pag. 91]{Mac}.\epf

The left $H\otk H^{\cop}$-comodule algebra $\diag(H)$ can be thought as
a coideal subalgebra in  $H\otk H^{\cop}$. The map $\iota:
\diag(H)\to  H\otk H^{\cop}$ given by $\iota(h)=h\_1\ot h\_2$ is an injective
comodule algebra map. Let $Q$ be the
coalgebra quotient $(H\otk H^{\cop})/ (H\otk H^{\cop}) \diag(H)^+$.

\begin{cor}\label{corol1} Let $K, S$ be left $H\otk H^{\cop}$-comodule algebras 
right $H\otk H^{\cop}$-simple with trivial coinvariants such that conditions
(i) and (ii) of Theorem \ref{invert-comod}  are fulfilled and
${}_S\Mo \boxtimes_{\Rep(H)}
 {}_K\Mo
 \simeq  \Rep(H).$ Then, there is
an isomorphism of $H\otk H^{\cop}$-comodule algebras
\begin{equation}\label{iso-cop1}
 S\Box_H K\simeq \End_{\diag(H)}(H\otk H^{\cop}\Box_Q V),
\end{equation}
for some $V\in {}^Q\Mo$. Moreover
\begin{equation}\label{iso-cop2}(S\Box_H K)^{\co H\otk H^{\cop}}=\End^Q(V).
\end{equation}

\end{cor}
\pf By Theorem  \ref{invert-comod} the module categories ${}_{S\Box_H K}\Mo$,
$ {}_{ \diag(H)}\Mo$ are equivalent. It follows from \cite[Lemma 1.26]{AM}
 that there exists an 
object $P\in {}^{H\otk H^{\cop}}\Mo_{ \diag(H)}$ such that
$$ S\Box_H K\simeq \End_{ \diag(H)}(P).$$ 
The left $H\otk H^{\cop} $-comodule structure on $\End_ {\diag(H)}(P)$ is given by
$\lambda:\End_{\diag(H)}(P)\to H\otk H^{\cop}\otk \End_{\diag(H)}(P)$, $\lambda(T)=T\_{-1}\ot
T\_0$ where
\begin{equation}\label{h-comod} \langle\alpha, T\_{-1}\rangle\,
T_0(p)=\langle\alpha, T(p\_0)\_{-1}\Ss^{-1}(p\_{-1})\rangle\,
T(p\_0)\_0,\end{equation} for any $\alpha\in (H\otk H^{\cop})^*$,
$T\in\End_{\diag(H)}(P)$, $p\in P$.

There is an equivalence of categories ${}^{H\otk H^{\cop}}\Mo_{ \diag(H)}
\simeq {}^Q\Mo$. The
functors $\Psi: {}^{H\otk H^{\cop}}\Mo_{ \diag(H)}\to {}^Q\Mo,$ $ \Phi:{}^Q\Mo\to {}^{H\otk H^{\cop}}\Mo_{ \diag(H)}$
defined by 
$$\Psi(M)=M/ (H\otk H^{\cop}) \diag(H)^+, \quad \Phi(V)= (H\otk H^{\cop})\Box_Q V,$$
$M\in {}^{H\otk H^{\cop}}\Mo_{ \diag(H)}$, $V\in {}^Q\Mo$ give an equivalence of categories. 
The left $H\otk H^{\cop}$-comodule structure on  $(H\otk H^{\cop})\Box_Q V$, $\delta:(H\otk H^{\cop})\Box_Q V
\to H\otk H^{\cop} \otk (H\otk H^{\cop})\Box_Q V$ and the right $ \diag(H)$-action are given by
$$\delta(h\ot t\ot v)=h\_1\ot  t\_2\ot h\_1\ot t\_1\ot v, \quad (h\ot t\ot v)\cdot x= hx\_1\ot tx\_2 \ot v, $$
for all $x\in H, h\ot t\ot v\in (H\otk H^{\cop})\Box_Q V$. This proves
isomorphism \eqref{iso-cop1}. 
Isomorphism \eqref{iso-cop2} follows from
$\End_{\diag(H)}(P)^{\co H}= \End^H_{\diag(H)}(P)$. 
\epf

\begin{cor}\label{corol2} Assume $H$ is pointed. Let $K, S$ be left $H\otk H^{\cop}$-comodule algebras
as in Corollary \ref{corol1}. Assume also that 
\begin{equation}\label{hy0} (S\Box_H K)_0= S_0 \Box_{H_0} K_0.
\end{equation}
Then ${}_{S_0}\Mo$, ${}_{K_0}\Mo$ are invertible $\Rep(H_0)$-bimodule categories.
\end{cor}
\pf By Corollary \ref{corol1}  there exists an object $V\in {}^Q\Mo$
such that 
$$S\Box_H K\simeq \End_{\diag(H)}(H\otk H^{\cop}\Box_Q V)\simeq 
\Hom_{\ku}(V, H\otk H^{\cop}\Box_Q V).$$
Let us explain the second isomorphism. The space
$\Hom_{\ku}(V, H\otk H^{\cop}\Box_Q V)$ is a left $H\otk H^{\cop}$-comodule
via $T\mapsto T\_{-1}\ot T\_0$ such that for all $ \alpha\in (H\otk H^{\cop})^*$
$$ \langle\alpha, T\_{-1}\rangle\,
T_0(v)=\langle\alpha, T(v\_0)\_{-1} \Ss^{-1}(v\_{-1})\rangle\,
T(v\_0)\_0, $$
for all $v\in V$. Recall that we are identifying
$\diag(H)$ with the coideal subalgebra $\iota(\diag(H))\subseteq H\otk H^{\cop}.$
There is an isomorphism $H\otk H^{\cop}\simeq Q\otk \diag(H)$
of right $\diag(H)$-modules and right $Q$-comodules \cite[Thm. 6.1]{S}.

Define $\phi: \End_{\diag(H)}(H\otk H^{\cop}\Box_Q V)
\to \Hom_{\ku}(V, H\otk H^{\cop}\Box_Q V)$, $\psi:\Hom_{\ku}(V, H\otk H^{\cop}\Box_Q V)
\to  \End_{\diag(H)}(H\otk H^{\cop}\Box_Q V)$ by
$$\phi(T)(v)=T(1\ot v), \quad  \psi(U)(h\ot v)=(h\ot 1)\cdot U(v),$$
for all $v\in V$, $h\in \diag(H)$. One can readily prove that $\phi$ and $\psi$
are one the inverse of each other and they are comodule morphisms.
Thus, there are isomorphisms
$$(S\Box_H K)_0 \simeq \Hom_{\ku}(V, H\otk H^{\cop}\Box_Q V)_0\simeq
\Hom_{\ku}(V_0, \widetilde{P}),$$
where $\widetilde{P}=\{\sum h\ot v\in H\otk H^{\cop}\Box_Q V: h\in H_0\otk H_0\}.$
Since there is an isomorphism $\widetilde{P}\simeq \diag(H_0)\otk V_0$ then
$$\Hom_{\ku}(V_0, \widetilde{P})\simeq \End_{\diag(H_0)}( \widetilde{P}),$$
which implies that the bimodule categories ${}_{(S\Box_H K)_0}\Mo$,
${}_{\diag(H_0)}\Mo$ are equivalent. By hypothesis
\eqref{hy0} the bimodule categories  
${}_{S_0 \Box_{H_0} K_0}\Mo$,
${}_{\diag(H_0)}\Mo$ are equivalent. Using Theorem \ref{invert-comod}
we get that both categories 
${}_{S_0}\Mo$, ${}_{K_0}\Mo$ are invertible $\Rep(H_0)$-bimodule categories.

\epf

Let $H$ be a pointed Hopf algebra such that the coradical is the group
algebra of an Abelian group $G$. Corollary \ref{corol2} tells us that
to find invertible $\Rep(H)$-bimodule categories we have to look
at those comodule algebras $K$ such that $K_0=\ku_{\psi} F$ where
$F\subseteq G$ is a subgroup, $\psi\in Z^2(F, \ku^{\times})$ is a 2-cocycle
such that the Morita class of the pair $(F,\psi)$ belongs to the Brauer-Picard
group of $\Rep(\ku G)$ that has been computed in \cite{ENO}.

\begin{rmk} In general there is an inclusion $ (S\Box_H K)_0\supseteq S_0 \Box_{H_0} K_0$.
Equality is not true for arbitrary comodule algebras, however
 \eqref{hy0} seems to be fulfilled in many examples
of comodule algebras over pointed Hopf algebras such that the bimodule
 categories are invertible.
\end{rmk}

\subsection{The Brauer-Picard group of $\Rep(G)$} 
In this Section we  compare the product of the Brauer-Picard group
of the category of representations of a finite Abelian group $G$ obtained in \cite{ENO}
and the product \eqref{pro-cot-a}.

\medbreak

Let $G$ be a finite Abelian group. The group $O(G\oplus \widehat{G})$ consists of group
 isomorphisms $\alpha:G\oplus \widehat{G}\to G\oplus \widehat{G} $ 
such that $\langle\alpha_2(g,\chi), \alpha_1(g,\chi) \rangle = \langle\chi, g\rangle$
for all $g\in G, \chi\in \widehat{G}$. Here $\alpha(g,\chi)= (\alpha_1(g,\chi) , \alpha_2(g,\chi))$.

\begin{teo}\cite[Corollary 1.2]{ENO} There is an isomorphism of groups 
$$\text{BrPic}(\Rep(G))\simeq O(G\oplus \widehat{G}).$$
\end{teo}\qed

Let $\alpha\in O(G\oplus \widehat{G})$ and define $U_\alpha\subseteq G\times G$ the subgroup
$$ L_\alpha=\{(\alpha_1(g,\chi) ,g): g\in G, \chi\in \widehat{G}\}.$$
and the 2-cocycle $\psi_\alpha:L_\alpha\times L_\alpha\to \ku^{\times}$ defined by
$$ \psi_\alpha((\alpha_1(g,\chi) ,g), (\alpha_1(h,\xi) ,h))=\langle\alpha_2(g,\chi)^{-1} , 
\alpha_1(h,\xi)\rangle \langle \chi,h\rangle.$$
It was proved in \cite{ENO} that the bimodule categories 
${}_{\ku_{\psi_\alpha}L_\alpha}\Mo$ are invertible.
\begin{prop} There is an equivalence of $\Rep(\ku G)$-bimodule categories
$${}_{\ku_{\psi_\alpha}L_\alpha \Box_{\ku G}\ku_{\psi_\beta}L_\beta}  \Mo
\simeq {}_{\ku_{\psi_{\alpha \beta}}L_{\alpha \beta}}\Mo.$$
\end{prop}
\pf It follows directly from Theorem \ref{invert-comod}.\epf

\begin{rmk} The product in $\text{BrPic}(\Rep(G))$ for a non-Abelian
group $G$ remains as an open problem. As pointed out by the referee to describe the elements and the product 
 in $\text{BrPic}(\Rep(G))$ one might have to use the description given
in \cite[Corollary 3.6.3]{Da}. 
\end{rmk}

\subsection*{Acknowledgment} This work was written in part during a research fe\-llowship granted by
CONICET, Argentina in the University of Hamburg, Germany. 
The author wants to thank the entire staff of Hamburg university and specially to
professor Christoph Schweigert,  Astrid D\"orh\"ofer and  Eva Kuhlmann for the warm hospitality.
Thanks are due to the referee for his careful reading and for pointing errors in a previous version of this
work.

\bibliographystyle{amsalpha}

\end{document}